\documentclass{amsart}
\usepackage{amsmath,amsthm,latexsym,amsfonts,amssymb,mathrsfs,array,manfnt}
\usepackage[all]{xy}
\usepackage{paralist}
\usepackage{framed}
\usepackage{cancel}
\usepackage{enumitem}
\usepackage{fullpage}
\usepackage{units}
\usepackage{gensymb}
\usepackage{setspace}
\usepackage{amscd}
\usepackage{tikz}
\usepackage{tikz-cd}
\usepackage{mathtools}
\usetikzlibrary{matrix}

\newcommand{\abs}[1]{\left\lvert#1\right\rvert}

\newcommand{\Z}{\ensuremath{\mathbb{Z}}}

\newcommand{\Q}{\ensuremath{\mathbb{Q}}}

\newcommand{\M}{\mathcal{M}}

\newcommand{\isom}{\cong} 

\renewcommand{\P}{\ensuremath{\mathbb{P}}}
\renewcommand{\H}{\ensuremath{\mathcal{H}}}

\renewcommand{\bar}[1]{\overline{#1}}

\renewcommand{\Im}{\ensuremath{\operatorname{Im}}}

\DeclareMathOperator{\Span}{Span}

\DeclareMathOperator{\Ker}{Ker}

\renewcommand{\Im}{\mathop{\mathrm{Im}}}

\DeclareMathOperator{\val}{Valence}

\DeclareMathOperator{\even}{even}
\DeclareMathOperator{\odd}{odd}

\usepackage{color}

\newcommand{\cA}{\mathcal{A}}

\newcommand{\cH}{\mathcal{H}}

\newcommand{\cK}{\mathcal{K}}

\newcommand{\cM}{\mathcal{M}}

\newcommand{\cQ}{\mathcal{Q}}
\newcommand{\cR}{\mathcal{R}}

\newcommand{\cV}{\mathcal{V}}
\newcommand{\cW}{\mathcal{W}}
\newcommand{\cX}{\mathcal{X}}

\newcommand{\bA}{\mathbf{A}}

\newcommand{\bE}{\mathbf{E}}
\newcommand{\bF}{\mathbf{F}}

\newcommand{\bK}{\mathbf{K}}

\newcommand{\bO}{\mathbf{O}}

\newcommand{\bS}{\mathbf{S}}

\theoremstyle{plain}
\newtheorem{theorem}{Theorem}
\numberwithin{theorem}{section}

\newtheorem{thm}[theorem]{Theorem}

\newtheorem{cor}[theorem]{Corollary}

\newtheorem{lemma}[theorem]{Lemma}

\newtheorem{lem}[theorem]{Lemma}

\theoremstyle{definition}
\newtheorem{Definition/Theorem}[theorem]{Definition/Theorem}
\newtheorem{Definition/Proposition}[theorem]{Definition/Proposition}

\newtheorem{Def}[theorem]{Definition}

\newtheorem{Corollary/Definition}[theorem]{Corollary/Definition}

\theoremstyle{remark}

\renewcommand{\H}{\cH}
\renewcommand{\M}{\cM}

\newcommand{\Mbar}{\bar{\cM}}

\renewcommand{\setminus}{\smallsetminus}

\newcommand{\bSP}{\mathbf{SP}}

\usetikzlibrary{arrows}
\usetikzlibrary{calc}

\usepackage{hyperref}

\begin{document}

\title{Pullbacks of $\kappa$ classes on $\Mbar_{0,n}$}
\author{Rohini Ramadas}
\email{rohini\_ramadas@brown.edu}
\address{Department of Mathematics\\Brown University\\Providence, RI}
\thanks{This work was partially supported by NSF grants
0943832, 1045119, 1068190, and 1703308.}
\subjclass[2010]{14H10 (primary), 14N99, 14M99, 20C30} 

\begin{abstract}
 
The moduli space $\Mbar_{0,n}$ carries a codimension-$d$ cycle class $\kappa_{d}$. We consider the subspace $\cK^{d}_{n}$ of $A^d(\Mbar_{0,n},\Q)$ spanned by pullbacks of $\kappa_d$ via forgetful maps. We find a permutation basis for $\cK^{d}_{n}$, and describe its annihilator under the intersection pairing in terms of $d$-dimensional boundary strata. As an application, we give a new permutation basis of the divisor class group of $\Mbar_{0,n}$.


\end{abstract}
\maketitle

\section{Introduction}\label{sec:Intro}

Mumford \cite{Mumford1983} introduced the tautological, codimension-$d$ class $\kappa_d$ in the cohomology/Chow group of the moduli space $\M_{g,n}$. This class extends to the moduli space $\Mbar_{g,n}$ of stable curves as well as to various partial compactifications of $\M_{g,n}$. Ring-theoretic relations involving $\kappa$ classes have been studied by Faber, Ionel, Pandharipande, Pixton, Zagier, Zvonkine, and several others, and play a role in the study of Gromov-Witten theory and mirror symmetry (\cite{faber_1999,Ionel2005,Pandharipande2012,PandharipandePixton2013,PPZ2016}; see \cite{Pandharipande2011, Pandharipande2016} for overviews).   

Here, we investigate $\kappa$ classes on $\Mbar_{0,n}$ from a linear-algebraic and representation-theoretic perspective. The symmetric group $S_n$ acts on $\Mbar_{0,n}$, and thus acts on its cohomology and Chow groups. Given any $T\subseteq\{1,\ldots,n\}$ with $\abs{T}\ge3$, there is a forgetful map $\pi_T:\Mbar_{0,n}\to\Mbar_{0,T}$. We set $\kappa_d^T:=\pi_T^*(\kappa_d)$, and consider the subspace $\cK^d_n\subseteq A^d(\Mbar_{0,n},\Q)$ spanned by $\{\kappa_d^T\}_{T\subseteq\{1,\ldots,n\}}$; this subspace is clearly $S_n$-invariant. Recall that a \textit{permutation basis} of a $G$-representation is one whose elements are permuted by the action of $G$. We show:

\smallskip

\noindent\textbf{Theorem A.} (Theorem \ref{thm:Main}\ref{it:KappaBasis}.) \textit{If $n\ge4$ and $1\le d\le n-3$, then $\cK^d_n$ has a permutation basis given by $\{\kappa_{d}^T\thickspace|\thickspace\abs{T}\ge(d+3), \abs{T}\equiv (d+3)\mod 2\}$.}

\smallskip


\subsection{Does $A^d(\Mbar_{0,n},\Q)$ have a permutation basis?} Getzler \cite{Getzler1995} and Bergstr\"om-Minabe \cite{BergstromMinabe2013} have given algorithms to compute the character of $A^d(\Mbar_{0,n},\Q)$ as an $S_n$-representation. It is not clear from these algorithms whether $A^d(\Mbar_{0,n},\Q)$ has a permutation basis. Farkas and Gibney \cite{FarkasGibney2003} have given a permutation basis for $A^{1}(\Mbar_{0,n},\Q)$. Theorem \ref{thm:Main} implies that $\cK^1_n=A^{1}(\Mbar_{0,n},\Q)$, so:

\smallskip

\noindent\textbf{Theorem B.} \textit{The set $\{\kappa_{1}^T\thickspace|\thickspace\abs{T}\ge4, \abs{T} \text{ even}\}$ is a permutation basis of $A^{1}(\Mbar_{0,n},\Q)$.}

\smallskip

The basis given by Theorem B is different from the one given in \cite{FarkasGibney2003}, which consists of certain boundary divisors and $\psi$ classes. For odd $n$, the two bases are isomorphic as $S_n$-sets, but for even $n$ they are not. 

Silversmith and the author \cite{RamadasSilversmith1} have produced a permutation basis for $A_2(\Mbar_{0,n},\Q)$, using Theorem \ref{thm:Main} as an ingredient. Very recent work of Castravet and Tevelev \cite{CastravetTevelev2020} on the derived category of $\Mbar_{0,n}$ implies that $A^*(\Mbar_{0,n},\Q)=\bigoplus_{d=0}^{n-3}A^{d}(\Mbar_{0,n},\Q)$ has a permutation basis; its elements, however, are not of pure degree. The question of whether or not $A^d(\Mbar_{0,n},\Q)$ has a permutation basis for all $d$ and $n$ remains open.  

\subsection{The dual story in $A_d(\Mbar_{0,n},\Q)$ and the proof of Theorem \ref{thm:Main}} There is an $S_n$-equivariant intersection pairing $A_d(\Mbar_{0,n},\Q)\times A^d(\Mbar_{0,n},\Q)\to\Q$. To prove Theorem A, we show:

\smallskip

\noindent\textbf{Theorem C.} (Theorem \ref{thm:Main}\ref{it:KappaPerpExactly},\ref{it:KappaTypeIPairingPerfect}.) If $n\ge4$ and $1\le d\le n-3$, we have
\begin{enumerate} 
\item \textit{The annihilator of $\cK^d_n\subseteq A^d(\Mbar_{0,n},\Q)$ is the subspace $\cV_{d,n}\subseteq A_d(\Mbar_{0,n},\Q)$ spanned by boundary strata whose dual trees have two or more vertices with valence at least four.}
\item \textit{$\cQ_{d,n}:=\frac{A_d(\Mbar_{0,n},\Q)}{\cV_{d,n}}$ is the dual of $\cK^d_n$.}
\end{enumerate}

\smallskip

It is straightforward to show that $\cV_{d,n}$ is contained in the annihilator of $\cK^d_n$, but to show equality involves a complicated induction on $n$. We use the fact that if $\pi$ denotes the forgetful morphism from $\Mbar_{0,n+1}$ to $\Mbar_{0,n}$, then:

\smallskip

\noindent\textbf{Theorem D.} (Theorem \ref{thm:Main}\ref{it:Sequences}) \textit{If $n\ge4$ and $1\le d\le n-3$, then we have the following (dual) exact sequences:}
    \begin{align*}
      &0\to\cQ_{d,n}\xrightarrow{\pi^*}{}\cQ_{d+1,n+1}\xrightarrow{\pi_*}{}\cQ_{d+1,n}\to0\\
       &0\to\cK^{d+1}_{n}\xrightarrow{\pi^*}{}\cK^{d+1}_{n+1}\xrightarrow{\pi_*}{}\cK^d_{n}\to0
    \end{align*}

\smallskip

It is difficult to use induction to study $A^d(\Mbar_{0,n},\Q)$, partly due to the fact that $A^d(\Mbar_{0,n},\Q)\xrightarrow{\pi^*}{}A^d(\Mbar_{0,n+1},\Q)\xrightarrow{\pi_*}{}A^{d+1}(\Mbar_{0,n},\Q)$ is not exact. This failure of exactness is also responsible for the fact that the dimensions of $A_d(\Mbar_{0,n},\Q)=A^{n-3-d}(\Mbar_{0,n},\Q)$ grow exponentially with $n$, whereas $\dim(\cK^{n-3-d}_n)$ grows as a degree-$d$ polynomial in $n$. 


\subsection{Significance for dynamics on $\M_{0,n}$} \textit{Hurwitz correspondences} are a class of multivalued dynamical systems on $\M_{0,n}$. They were introduced by Koch \cite{Koch2013} in the context of Teichm\"uller theory and complex dynamics on $\P^1$, and their dynamics were studied by the author \cite{Ramadas2015, Ramadas2016, Ramadas2019}. A Hurwitz correspondence $\H$ on $\M_{0,n}$ induces a linear pushforward action on $\cQ_{d,n}$, and the $d$-th dynamical degree of $\H$ (a numerical invariant of algebraic dynamical systems) is the largest eigenvalue of this action \cite{Ramadas2015}. Theorem \ref{thm:Main} can be used to re-interpret Theorem 10.6 of \cite{Ramadas2015} to conclude that $\H$ acts on pullbacks of $\kappa$ classes, and that this action encodes important information about the dynamics of $\H$:  

\smallskip

\noindent\textbf{Theorem E.} \textit{Suppose $\H$ is a Hurwitz correspondence on $\M_{0,n}$. If $1\le d\le n-3$, then  $\cK^d_n$ is invariant under the pullback $\H^*:A^d(\Mbar_{0,n},\Q)\to A^d(\Mbar_{0,n},\Q)$, and the $d$-th dynamical degree of $\H$ is the largest eigenvalue of the action of $\H^*$ on $\cK^d_n$.}

%


\subsection*{Acknowledgements}
I am grateful to David Speyer, Rob Silversmith, and Renzo Cavalieri for useful conversations. Rob Silversmith noticed a patten in my experimental data: the dimensions of $\cQ_{d,n}$ can be expressed as sums of binomial coefficients. This observation led me to conjecture the correct dual basis of $\cQ_{d,n}$. My initial expression of the dual basis elements was purely combinatorial; Renzo Cavalieri observed that these elements could be expressed as functionals induced by pairing with kappa classes.

\subsection*{Notation and conventions} For $n$ a positive integer, we denote by $[n]$ the set $\{1,\ldots,n\}$. For $\bA$ a finite set, we denote by $\Q\bA$ the free $\Q$-vector space on $\bA$. For $\cV$ a vector space, we denote by $\cV^{\vee}$ its dual. For a linear map $\mu:\cV\to\cX$, we denote by $\mu^{\vee}$ its dual map. If $\cW$ is a subspace of $\cV$, we denote by $\cW^\perp$ its annihilator in $\cV^{\vee}$.

\section{Cycle classes on $\Mbar_{0,n}$}\label{sec:(co)homology}

The Chow group $A_{d}(\Mbar_{0,n})$ is a finitely generated free abelian group generated, though not freely, by the fundamental classes of $d$ dimensional boundary strata \cite{Keel1992, KontsevichManin1994}. We set $\cA_{d,n}:=A_{d}(\Mbar_{0,n},\Q)$, and $\cA^{d}_n:=\cA_{n-3-d,n}$. There is an $S_n$-equivariant non-degenerate intersection pairing $\cA_{d,n}\times \cA^d_n\to \Q$; this identifies $\cA^d_n$ with $\cA_{d,n}^{\vee}$.  

A stable $n$-marked tree is a tree $\sigma$ with $n$ marked legs such that every vertex has valence at least $3$ (counting the legs). Boundary strata on $\Mbar_{0,n}$ are in bijection with stable $n$-marked trees; if $\sigma$ is a stable $n$-marked tree we denote by $X_\sigma$ the corresponding boundary stratum on $\Mbar_{0,n}$. Boundary strata are isomorphic to products of smaller moduli spaces: $X_\sigma\isom \prod\Mbar_{0,\val(v)}$, where the product is over vertices $v$ of $\sigma$. We conclude that $X_\sigma$ is positive-dimensional if and only of its dual tree has at least one vertex with valence at least four. If $\sigma$ has exactly one vertex $v$ with valence at least four, then $X_\sigma$ is isomorphic to $\Mbar_{0,n'}$, where $n'=\val(v)$, since the factors in the above product decomposition of $X_\sigma$ corresponding to vertices other than $v$ are all isomorphic to single-point spaces.

\begin{Def}
We say that a positive-dimensional boundary stratum $X_\sigma$ is \textit{Type I} if its dual tree $\sigma$ has  exactly one vertex with valence at least four; in this case we also say that $\sigma$ is a Type I stable tree. We say that a positive-dimensional boundary stratum $X_\sigma$ is \textit{Type II} if its dual tree $\sigma$ has two or more vertices with valence at least four; in this case we also say that $\sigma$ is a type II stable tree. For $n\ge4$ and $d=1,\ldots, n-4$, we set $\cV_{d,n}\subset \cA_{d,n}$ to be the subspace generated by the fundamental classes of Type II boundary strata. We set $\cQ_{d,n}$ to be the quotient $\cA_{d,n}/\cV_{d,n}$. Note that since $\cV_{d,n}$ is $S_n$-invariant, $\cQ_{d,n}$ inherits an action of $S_n$. Also note that $\cQ_{d,n}$ is generated by the fundamental classes of Type I boundary strata. 
\end{Def}

\begin{Def}
Suppose $\sigma$ is stable $n$-marked tree and $v$ a vertex on $\sigma$. We obtain from the pair $(\sigma, v)$ a set partition $\Pi_*(\sigma,v)$ of $[n]$ as follows: $i$ and $j$ are in the same part of $\Pi_*(\sigma, v)$ if and only if the $i$- and $j$-marked legs on $\sigma$ are on the same connected component of $\sigma\setminus \{v\}$. Note that the number of parts of $\Pi_*(\sigma, v)$ equals the valence of $v$. If $\sigma$ is Type I and $v$ is its unique vertex with valence at least four, then the partition $\Pi_*(\sigma,v)$ is intrinsically associated to $\sigma$, so we denote it by $\Pi_*(\sigma)$. In this case we have $\dim(X_{\sigma})=\abs{\Pi_*(\sigma)}-3(=\val(v)-3)$. 
\end{Def}

\begin{lem}\label{lem:PartitionDeterminesClass}
Suppose that $\sigma_1$ and $\sigma_2$ are two Type I stable $n$-marked trees, and suppose $\Pi_*(\sigma_1)=\Pi_*(\sigma_2)$. Then $[X_{\sigma_1}]=[X_{\sigma_2}]\in \cA_{d,n}$, where $d=\dim(X_{\sigma_1})=\dim(X_{\sigma_2})$. 
\end{lem}
\begin{proof} This follows immediately from the fact that $\sigma_1$ and $\sigma_2$ differ only in the arrangement of trivalent subtrees. See Lemma 5.2.1 of \cite{RamadasThesis} for a detailed proof.
%
\end{proof} 

For $n\ge 1$ and $d\ge -3$, we set $\bSP_{d,n}$ to be the set of all set partitions of $[n]$ having exactly $d+3$ parts. By Lemma \ref{lem:PartitionDeterminesClass}, if $n\ge 4$ and $d\ge 1$, then there is a well-defined map $\bSP_{d,n}\to \cA_{d,n}$ sending $\Pi$ to $[X_{\sigma}]$, where $\sigma$ is any Type I stable $n$-marked tree such that $\Pi=\Pi_*(\sigma)$ (it is clear that such a $\sigma$ exists). Extending by linearity and composing with the quotient map from $\cA_{d,n}$ to $\cQ_{d,n}$, we obtain a surjective, $S_n$-equivariant, linear map $\Q\bSP_{d,n}\to \cQ_{d,n}$. 

\begin{Def}\label{Def:REL}
  For $n\ge 4$ and $1\le d\le n-4,$ let $\cR_{d,n}$ be the subspace of $\Q\bSP_{d,n}$ generated by elements of the form:
  \begin{align*}
    &\{P_1\cup P_2,P_3,P_4,\ldots,P_{d+4}\}+\{P_1,P_2,P_3\cup
    P_4,\ldots,P_{d+4}\}\\
&\quad-\{P_1\cup
    P_3,P_2,P_4,\ldots,P_{d+4}\}-\{P_1,P_3,P_2\cup P_4,\ldots,P_{d+4}\}, 
  \end{align*}
  where $\{P_1,P_2,P_3,P_4,\ldots,P_{d+4}\}$ is a set partition of $[n]$ with $d+4$ parts.  Note that $\cR_{d,n}$ is $S_n$-invariant. 
 \end{Def}
 
\begin{lem}
The kernel of the surjective linear map $\Q\bSP_{d,n}\to \cQ_{d,n}$ is $\cR_{d,n}$. 
\end{lem}
\begin{proof}
This follows immediately from the description of linear relations among boundary strata given in Theorem 7.3 of \cite{KontsevichManin1994}, or see Lemma 5.2.4 of \cite{RamadasThesis} for a detailed proof.
\end{proof}
We conclude that $\cQ_{d,n}$ is naturally isomorphic, as an $S_n$ representation, to $\Q\bSP_{d,n}/\cR_{d,n}$. We use this identification throughout the paper.

\subsection{Kappa classes and the intersection pairing}

For $d=0,\ldots, n-3$, $\Mbar_{0,n}$ carries a codimension $d$ kappa class $\kappa_d\in \cA^d_n$, see \cite{ArbarelloCornalba1996, ArbarelloCornalba1998} for the definition and properties.

\begin{lemma}\label{lem:pairkappawithstrata}
Let $n\ge4$ and $1\le d\le n-3$. Suppose $X_\sigma$ is a dimension $d$ boundary stratum on $\Mbar_{0,n}$. Then 
\begin{align}
[X_\sigma]\cdot\kappa_d=
\begin{cases}
0&\text{$X_\sigma$ is Type II}\\
1&\text{$X_\sigma$ is Type I}
\end{cases}
\end{align}
\end{lemma}

\begin{proof}
This follows by a standard computation from Equation 1.8 of \cite{ArbarelloCornalba1996} together with Lemma 1.1 (12) of \cite{CavalieriYang2011}.
\end{proof}

\begin{Def}
For $T\subseteq [n]$, set $\kappa_d^T$ to be the pullback, to $\Mbar_{0,n}$, of the codimension $d$ kappa class on $\Mbar_{0,T}$, via the natural forgetful map $\pi_T:\Mbar_{0,n}\to \Mbar_{0,T}$. 
\end{Def}

\begin{Def}
Set $\mathcal{K}^d_n$ to be the subspace of $\mathcal{A}^d_n$ spanned by the classes $\{\kappa_T\thickspace|\thickspace T\subseteq [n]\}$. Note that $\mathcal{K}^d_n$ is $S_n$-invariant.
\end{Def}

\begin{lemma}\label{lem:pairkappapullbackswithstrata}
Let $n\ge4$ and $1\le d\le n-3$. Suppose $X_\sigma$ is a dimension $d$ boundary stratum on $\Mbar_{0,n}$, and $T\subseteq [n]$. Then 
\begin{align}
[X_\sigma]\cdot\kappa_d^T=
\begin{cases}
0&\text{$X_\sigma$ is Type II}\\
1&\text{$X_\sigma$ is Type I and $\forall P\in\Pi_*(\sigma), \thickspace P\cap T\ne\emptyset$}\\
0&\text{$\sigma$ is Type I and $\exists P\in\Pi_*(\sigma)$ s.t. $P\cap T=\emptyset$}
\end{cases}
\end{align}
\end{lemma}
\begin{proof}
By the projection formula, $[X_\sigma]\cdot\kappa_d^T=(\pi_T)_*([X_\sigma])\cdot \kappa_d$. By \cite{Ramadas2015}, if $X_{\sigma}$ is Type II, then $\pi_*([X_{\sigma}])$ is either zero, or the fundamental class of a Type II boundary stratum of $\Mbar_{0,T}$. If $X_{\sigma}$ is Type I, then if $\exists P\in\Pi_*(\sigma)$ s.t. $P\cap T=\emptyset$, then $\pi_*([X_{\sigma}])=0$, while if $\forall P\in\Pi_*(\sigma), \thickspace P\cap T\ne\emptyset$, then $\pi_*([X_{\sigma}])$ is the fundamental class of a Type I boundary stratum of $\Mbar_{0,T}$. Applying  Lemma \ref{lem:pairkappawithstrata}, we obtain the desired result.
\end{proof}

\begin{cor}\label{cor:kappaandQduals}
\begin{enumerate}
\item The subspace $\cK_{n}^d\subseteq\cA_{n}^d$ is orthogonal, with respect to the intersection pairing, to $\cV_{d,n}\subseteq\cA_{d,n}$, i.e. we have $\cK_n^d\subseteq \cV_{d,n}^{\perp}$.\label{it:KappaPerp}
\item The intersection pairing on $\Mbar_{0,n}$ descends to a pairing $\cQ_{d,n}\times\cK_{n}^d\to\Q$. \label{it:KappaTypeIPairing}
\end{enumerate}
\end{cor}
We will eventually show that $\cK_n^d=\cV_{d,n}^{\perp}$, which implies that $\cK_n^d=\cQ_{d,n}^{\vee}$. 
\begin{Def}\label{Def:Pairing}
 Define a pairing
   $ \langle\cdot,\cdot\rangle:\{\text{set partitions
    of $[n]$}\}\times\{\text{subsets of $[n]$}\}\to\Z$.
 For a set partition $\Pi$ and subset $T$, set
  \begin{align*}
    \langle\Pi,T\rangle=
    \begin{cases}
      1&\text{$\forall P\in\Pi, \quad P\cap T\ne\emptyset$}\\
      0&\exists P\in\Pi \text{ s.t. }P\cap T=\emptyset
    \end{cases}
  \end{align*}
 
\end{Def}
 From Lemma \ref{lem:pairkappapullbackswithstrata} we obtain a purely combinatorial description of the pairing between $\cQ_{d,n}$ and $\cK^d_n$ obtained in Corollary \ref{cor:kappaandQduals}:
 
 \begin{lem}\label{lem:combinatorialpairing}
  For any Type I $d$-dimensional boundary stratum $X_\sigma$, and for any subset $T\subseteq[n]$, we have that $[X_\sigma]\cdot \kappa^T_d=\langle\Pi_*(\sigma),T\rangle$. 
  \end{lem}

\subsection{Pushing forward and pulling back via forgetful maps}

The forgetful morphism $\pi:\Mbar_{0,n+1}\to\Mbar_{0,n}$ induces pushforward maps $\pi_*:\cA_{d,n+1}\to\cA_{d,n}$ and pullback maps $\pi^*:\cA_{d,n}\to\cA_{d+1,n+1}$. By \cite{Ramadas2015}, $\pi_*(\cV_{d,n+1})\subseteq \cV_{d,n}$. Also, if $X_\sigma$ is a Type II boundary stratum, then $\pi^*([X_{\sigma}])$ is a sum of fundamental classes of Type II boundary strata of $\Mbar_{0,n+1}$. This implies that $\pi^*(\cV_{d,n})\subseteq \cV_{d+1,n+1}$. Thus there are induced pushforward maps $\pi_*:\cQ_{d,n+1}\to\cQ_{d,n}$ and pullback maps $\pi^*:\cQ_{d,n}\to\cQ_{d+1,n+1}$. 

\begin{lem}
\begin{enumerate}
\item The pushforward $\pi_*:\cQ_{d,n+1}\to\cQ_{d,n}$ lifts to $\tilde{\pi}_*:\Q\bSP_{d,n+1}\to\Q\bSP_{d,n}$ where 
\begin{align*}
\tilde{\pi}_*(\{P_1,\ldots,P_{d+3}\})=&
                          \begin{cases}
                            0&\exists i \text{ s.t. }P_i=\{n+1\}\\
                            \{P_1\setminus\{n+1\},\ldots,P_{d+3}\setminus\{n+1\}\}&\text{otherwise}
                          \end{cases}
\end{align*}
\item The pullback $\pi^*:\cQ_{d,n}\to\cQ_{d+1,n+1}$ lifts to $\tilde{\pi}^*:\Q\bSP_{d,n}\to\Q\bSP_{d+1,n+1}$ where $\tilde{\pi}^{*}(\Pi)=\Pi\cup\{\{n+1\}\}$, for $\Pi\in\bSP_{d,n}$
\end{enumerate}
\end{lem}
\begin{proof}[Sketch of proof]
This lemma follows from the observation that if $X_\sigma$ is a Type I $d$-dimensional boundary stratum whose dual tree $\sigma$ has exactly one vertex with valence at least four, then 
\begin{align}\label{eq:pullbackstratum}
\pi^{*}([X_\sigma])=[X_{\sigma'}]+(\text{sum of classes of Type II boundary strata}),
\end{align}
 where $\sigma'$ is a Type I stable $(n+1)$-marked tree with the property that $\Pi_*(\sigma')=\Pi_*(\sigma)\cup\{\{n+1\}\}$.
\end{proof}

The pushforward maps $\pi_*:\cA_{d,n+1}\to\cA_{d,n}$ and $\pi_*:\cQ_{d,n+1}\to\cQ_{d,n}$ are easily seen to be surjective. Since $\pi$ has positive relative dimension (equal to one), $\pi_*\circ\pi^*=0$ on $\cA_{*,*}$, thus also on $\cQ_{*,*}$.

\begin{lem}\label{lem:exactinmiddle}
  For $n\ge 4$ and $k\ge 1$, the complex $\cQ_{d,n}\xrightarrow{\pi^*}{}\cQ_{d+1,n+1}\xrightarrow{\pi_*}{}\cQ_{d+1,n}$ is exact.
\end{lem}
\begin{proof} We use the lifts of $\pi^*$ and
  $\pi_*$ to $\tilde{\pi}^*:\Q\bSP_{d,n}\to \Q\bSP_{d+1,n+1}$ and
  $\tilde{\pi}_*:\Q\bSP_{d+1,n+1}\to \Q\bSP_{d+1,n}$ respectively. We have:
  
  $$\Ker(\tilde{\pi}_*)=\Im(\tilde{\pi}^*)
  +\Span(\{P_1\cup\{n+1\},P_2,P_3,P_4,\ldots\}-\{P_1,P_2\cup\{n+1\},P_3,P_4,\ldots\}\}),$$ 
  
and: 

\begin{align*}
  &\{P_1\cup\{n+1\},P_2,P_3,P_4,\ldots\}-\{P_1,P_2\cup\{n+1\},P_3,P_4,\ldots\}\\
  =&\left(\{P_1\cup\{n+1\},P_2,P_3,P_4,\ldots\} +\{P_1,\{n+1\},P_2\cup
    P_3,P_4,\ldots\}\right.\\&\quad\quad\left.-\{P_1,P_3,P_2\cup\{n+1\},P_4,\ldots\}-\{P_1\cup
    P_3,P_2,\{n+1\},P_4,\ldots\}\right)\\
  &\quad\quad-\left(\{P_1,\{n+1\},P_2\cup
    P_3,P_4,\ldots\}-\{P_1\cup P_3,P_2,\{n+1\},P_4,\ldots\}\right)\in \cR_{k+1,n+1}+\Im(\tilde{\pi}^*)
\end{align*}
This implies that  $\Ker(\tilde{\pi}_*)=\Im(\tilde{\pi}^*)+\cR_{k+1,n+1}$, which in turn implies that $\Ker(\pi_*)=\Im(\pi^*)$.
\end{proof}



The pullback $\pi^*:\cA^d_n\to\cA^d_{n+1}$, which by the projection formula is dual to $\pi_*:\cA_{d,n}\to\cA_{d,n+1}$, restricts to $\pi^*:\cK^d_n\to\cK^d_{n+1}$, and sends $\kappa_d^T$ on $\Mbar_{0,n}$ to $\kappa_d^T$ on $\Mbar_{0,n+1}$. The pushforward  $\pi_*:\cA^{d+1}_{n+1}\to\cA^d_{n}$ is dual to $\pi^*:\cA_{d+1,n+1}\to\cA_{d,n}$.

\begin{lem}
The pushforward  $\pi_*:\cA^{d+1}_{n+1}\to\cA^d_{n}$ restricts to $\pi_*:\cK^{d+1}_{n+1}\to\cK^d_{n}$, with 
\begin{align*}
\pi_* (\kappa_d^T)  =   \begin{cases}
                            0& n+1\in T\\
                            \kappa_d^T & n+1\not\in T
                          \end{cases}
\end{align*}
\end{lem} 
\begin{proof}
By Lemma \ref{cor:kappaandQduals}, \ref{it:KappaPerp}, we have that $\forall d,n, \cK^d_n\subseteq \cV_{d,n}^{\perp}$.  Since $\pi^*(\cV_{d,n})\subseteq \cV_{d+1,n+1}$, and since, by the projection formula, $\pi_*$ and $\pi^*$ are dual maps, we have that $\pi_*(\cK^{d+1}_{n+1})\subseteq \cV_{d,n}^{\perp}$. This means that for $T\subseteq[n+1]$, the class $\pi_*(\kappa_d^T)$ is determined by the functional that it defines on $\cQ_{d,n}$, i.e. by the values of $[X_\sigma]\cdot \pi_*(\kappa_d^T)$, where $X_\sigma$ ranges over all Type I $d$-dimensional boundary strata on $\Mbar_{0,n}$. Given such an $X_\sigma$, we have, by the projection formula, by the expression for $\pi^*([X_\sigma])$ given in Equation \ref{eq:pullbackstratum}, and by applying Lemma \ref{lem:combinatorialpairing} twice, that 
\begin{align*}
[X_\sigma]\cdot \pi_*(\kappa_d^T)=\pi^{*}([X_\sigma])\cdot \kappa_d^T &=\langle\Pi_*(\sigma)\cup\{\{n+1\}\},T\rangle\\
&=  \begin{cases}
      1&\text{if $\forall P\in\Pi_*(\sigma)\cup\{\{n+1\}\}, P\cap T\ne\emptyset$}\\
      0&\text{if $\exists P\in\Pi_*(\sigma)\cup\{\{n+1\}\}$ s.t. $P\cap T=\emptyset$} 
    \end{cases}\\
&          =  \begin{cases}
      1&\text{if $\forall P\in\Pi_*(\sigma), P\cap T\ne\emptyset$, and $n+1\in T$}\\
      0&\text{if $\exists P\in\Pi_*(\sigma)$ s.t. $P\cap T=\emptyset $, or if $n+1\not\in T$}
    \end{cases}\\
                &=  \begin{cases}
                \langle\Pi_*(\sigma),T\rangle &\text{if $n+1\in T$}\\
                0&\text{if $n+1\not\in T$}
                \end{cases}\\
                 &=  \begin{cases}
 [X_\sigma]\cdot \kappa_d^T  & \text{if $n+1\in T$}\\
                0&\text{if $n+1\not\in T$}
                \end{cases}
\end{align*}
\end{proof}

\section{Main results and proofs}\label{sec:Proofs}

\subsection{The set-up}\label{sec:setup}
Throughout Section \ref{sec:Proofs}, we use the identification $\cQ_{d,n}=\Q\bSP_{d,n}/\cR_{d,n}$ introduced in Section \ref{sec:(co)homology}; we write an element of $\cQ_{d,n}$ as a $\Q$-linear combination of set partitions of $[n]$ with $d+3$ parts, rather than as a $\Q$-linear combination of fundamental classes of Type I boundary strata. 

\begin{Def}
For $n\ge1$ and $d\ge -3$, we set $\bK^d_n$ to be the set $\{T\subseteq[n]\thickspace|\thickspace\abs{T}\ge(d+3), \abs{T}\equiv (d+3)\mod 2\}$.
\end{Def}

\begin{Def}\label{Def:bKintocK}
There is a natural linear map $\psi_{d,n}:\Q\bK^d_n\to\cK^d_n$ sending $T$ to $\kappa_d^T$. 
\end{Def}

The map $\psi_{d,n}$, together with the intersection pairing $\cQ_{d,n}\times\cK_{n}^d\to\Q$ induces the pairing $\cQ_{d,n}\times\Q\bK^d_n\to\Q$, where, for $\Pi\in\bSP_{d,n}$ and $T\in \bK^d_n$, $\Pi\cdot T=\langle\Pi,T\rangle$ as in Definition \ref{Def:Pairing}.  
Note that if $\bS$ is an $S_n$-set, then $\Q\bS$ is canonically and $S_n$-equivariantly isomorphic to its dual. Thus $\Q\bK^d_n=(\Q\bK^d_n)^{\vee}$. 

\begin{Def}\label{def:phi}
For $d\ge -1$, the pairing $\langle.,.\rangle$ (Definition \ref{Def:Pairing}) between set partitions and subsets of $[n]$ induces a map $\tilde{\phi}_{d,n}:\Q\bSP_{d,n}\to (\Q\bK^{d}_{n})^{\vee}=\Q\bK^{d}_{n}$. For $d\ge1,$ $\tilde{\phi}_{d,n}$ descends to a map $\phi_{d,n}:\cQ_{d,n}\to \Q\bK^{d}_{n}$. The map $\tilde{\phi}_{d,n}$ on generators $\Pi\in\bSP_{d,n}$ is given explicitly by:
\begin{align*}
  \tilde{\phi}_{d,n}(\Pi)=\sum_{T\in\bK^{d}_{n}}\langle\Pi,T\rangle \cdot T.
\end{align*}
\end{Def}

%

\begin{Def}\label{Def:AlphaBeta} Define maps $\alpha:\Q\bK^{d}_{n}\to \Q\bK^{d+1}_{n+1}$ and $\beta:\Q\bK^{d}_{n+1}\to \Q\bK^{d}_{n}$, where:
\begin{align*}
\alpha(T)&= T\cup\{n+1\}& \beta(T)&= \begin{cases}
       T&n+1\not\in T\\
       0&n+1\in T
     \end{cases}
\end{align*}
\end{Def}

\begin{lem}\label{lem:exactsequence}
The following is an exact sequence:
\begin{align}
0\to\Q\bK^{d}_{n}\xrightarrow{\alpha}{}\Q\bK^{d+1}_{n+1}\xrightarrow{\beta}{}\Q\bK^{d+1}_{n}\to0\label{eq:seqbK}
\end{align}  
\end{lem}
\begin{proof}[Sketch of proof]
Observe that $\Im(\alpha)=\Ker(\beta)=\Q\{T\in\bK^{d+1}_{n+1}\thickspace|\thickspace n+1\in T\}$.
\end{proof}

\begin{lem}\label{lem:alphabetacommutewithpushpull}
The following diagram commutes:
\begin{center}
    \begin{tikzcd}
      &\Q\bSP_{d,n}\arrow[r,"\tilde{\pi}^*"]\arrow[d,"\tilde{\phi}_{d,n}"]&\Q\bSP_{d+1,n+1}\arrow[r,"\tilde{\pi}_*"]\arrow[d,"\tilde{\phi}_{d+1,n+1}"]&\Q\bSP_{d+1,n}\arrow[d,"\tilde{\phi}_{d+1,n}"]\\
      &\Q\bK^{d}_{n}\arrow[r,"\alpha"]&\Q\bK^{d+1}_{n+1}\arrow[r,"\beta"]&\Q\bK^{d+1}_{n}
    \end{tikzcd}
  \end{center}
\end{lem}
\begin{proof}
\noindent\textbf{Commutativity of the left square:}
Given $\Pi\in\bSP_{d,n}$, we have
\begin{align*}
\phi_{d+1,n+1}(\tilde{\pi}^{*}(\Pi))=\sum_{T\in\bK^{d+1}_{n+1}}\langle\Pi\cup\{\{n+1\}\},T\rangle \cdot T&=\sum_{\substack{T\in\bK^{d+1}_{n+1}\\n+1\in T}}\langle\Pi\cup\{\{n+1\},T\rangle \cdot T\\
&=\sum_{\substack{T'\in\bK^{d}_{n}}}\langle\Pi,T'\rangle \cdot T'\cup\{n+1\}\\
&=\alpha(\sum_{\substack{T'\in\bK^{d}_{n}}}\langle\Pi,T'\rangle \cdot T')=\alpha(\phi_{d,n}(\Pi))
\end{align*} 

\noindent\textbf{Commutativity of the right square:}
Given $\Pi'=\{P_1,\ldots,P_{d+4}\}\in\bSP_{d+1,n+1}$, we may assume without loss of generality that $n+1\in P_1$. There are two cases:

\noindent\textbf{Case 1:} $P_1=\{n+1\}$. Then $\tilde{\pi}_{*}(\Pi')=0$ so $\phi_{d+1,n}(\tilde{\pi}_{*}(\Pi'))=0$. Note that for $T\subset [n+1]$, we have that if $n+1\not\in T$, then $\langle\Pi',T\rangle=0$, while if $n+1\in T$, then $\beta(T)=0$. This implies that 
\begin{align*}
\beta(\phi_{d+1,n+1}(\Pi'))&=\beta(\sum_{T\in\bK^{d+1}_{n+1}}\langle\Pi',T\rangle \cdot T)=\beta(\sum_{\substack{T\in\bK^{d+1}_{n+1}\\n+1\in T}}\langle\Pi',T\rangle \cdot T)=0,
\end{align*}

\noindent\textbf{Case 2:} $P_1\ne\{n+1\}$. Then 

\begin{align*}
\phi_{d+1,n}(\tilde{\pi}_{*}(\Pi'))&=\sum_{T'\in\bK^{d+1}_{n}}\langle\{P_1\setminus\{n+1\},P_2,\ldots,P_{d+4}\},T'\rangle \cdot T'\\
&=\sum_{\substack{T'\in\bK^{d+1}_{n}\\T'\cap P_1\setminus\{n+1\}\ne\emptyset\\T'\cap P_2,\ldots,T'\cap P_{d+1}\ne\emptyset}} T'=\sum_{\substack{T'\in\bK^{d+1}_{n+1}\\n+1\not\in T\\\langle\Pi',T\rangle=1}} T'=\beta(\phi_{d+1,n+1}(\Pi'))
\end{align*}
\end{proof}

We use the following lemma several times; its proof follows from a standard diagram chase. 

\begin{lem}\label{lem:fourvariant}[Variant of the Four Lemma]
Suppose we have a commutative diagram of vector spaces as follows
  \begin{center}
 \begin{tikzcd}
      &\cW_1\arrow[r,"f_1"]\arrow[d,"h_1"]&\cW_2\arrow[r,"f_2"]\arrow[d,"h_2"]&\cW_3\arrow[d,"h_3"]\arrow[r]&0\\
      &\cX_1\arrow[r,"g_1"]&\cX_2\arrow[r,"g_2"]&\cX_3 
    \end{tikzcd}
  \end{center}
Suppose further that the bottom row is exact at $\cX_2$, that the top row is exact at $\cW_3$, and that $h_1$ and $h_3$ are surjective. Then $h_2$ is surjective. 
\end{lem} 
\subsection{A preliminary lemma}
In this section, we prove some technical results --- Lemmas \ref{lem:evenoddsurjective}, \ref{lem:n2surjective} and  \ref{lem:n3surjective}. These are not of independent interest, but are necessary to prove Theorem \ref{thm:root} in Section \ref{sec:Main}. The proofs (and statements) of these three lemmas are conceptually similar to each other, as well as to those of Theorem \ref{thm:root}; all four proofs use the Four Lemma to induct on $n$. Lemma \ref{lem:evenoddsurjective} is required in the inductive step of Lemma \ref{lem:n2surjective}, which is required in the inductive step of Lemma \ref{lem:n3surjective}, which in turn is required in the inductive step of Theorem \ref{thm:root}. The proofs of Lemmas \ref{lem:evenoddsurjective} and \ref{lem:n2surjective} also involve some intricate combinatorics of set partitions and subsets.

\begin{Def}\label{def:EvenOdd} For $n>0$, we set:
\begin{align*}
\bE_{n}&:= \{T\subseteq[n]|\abs{T} \text{ even}\}; & \bO_{n}&:=\{T\subseteq[n]|\abs{T}\text{ odd}\};& \bF_{n}&:=\{(P_1,P_2)|P_1\cup P_2=[n],P_1\cap P_2=\emptyset,1\in P_1\}
  \end{align*}
\end{Def}

Note that $\bF_{n}\setminus\{([n],\emptyset)\}$ is in canonical bijection with $\bSP_{-1,n}$, so $\Q\bF_n$ is canonically isomorphic to $\Q\{([n],\emptyset)\}\oplus\Q\bSP_{-1,n}$. There are maps $\alpha:\Q\bE_n\to \Q\bO_{n+1}$, $\alpha:\Q\bO_n\to \Q\bE_{n+1}$, $\beta:\Q\bE_{n+1}\to \Q\bE_n$ and $\beta:\Q\bO_{n+1}\to \Q\bO_n$, analogous to the maps $\alpha$ and $\beta$ as in Definition \ref{Def:AlphaBeta}. Define maps $\odd_n:\Q\bF_{n}\to \Q\bO_{n}$ and $\even_n:\Q\bF_{n}\to \Q\bE_{n}$, where
\begin{align*}
\odd_n((P_1,P_2))&=\sum_{\substack{T\subseteq P_1\\\text{$\abs{T}$
  odd}}}(-T)+\sum_{\substack{T\subseteq P_2\\\text{$\abs{T}$
  odd}}}T; &
 \even_n((P_1,P_2))&=\sum_{\substack{T\subseteq P_1\\\text{$\abs{T}$
  even}}}(-T)+\sum_{\substack{T\subseteq P_2\\\text{$\abs{T}$
  even}}}(-T)
\end{align*}

\begin{lem}\label{lem:evenoddsurjective}
 For $n\ge1$, the maps $\odd_n$ and $\even_n$ are surjective.
\end{lem}
\begin{proof}
  We induct on $n$. \textbf{Base case:} $n=1$. We have:
  \begin{align*}
    \bF_{1}&=\{(\{1\},\emptyset)\}; &  \bE_1&=\{\emptyset\}; & \bO_1&=\{\{1\}\}; &   \odd_1((\{1\},\emptyset))&=-\{1\};&  \even_1((\{1\},\emptyset))&=-2\emptyset.
  \end{align*}
  This establishes the base case. 
  
  \noindent \textbf{Inductive hypothesis:} The proposition holds up to some $n\ge1$. 
  
  \noindent \textbf{Inductive step:} Define maps $\tilde{\pi'}_*:\Q\bF_{n+1}\to \Q\bF_{n}$ and $\gamma:\Q\bF_{n}\to \Q\bF_{n+1}$, where:
  \begin{align*}
    \tilde{\pi'}_*((P_1,P_2))&=(P_1\setminus\{n+1\},P_2\setminus\{n+1\})\\
 \gamma((P_1,P_2))&=(P_1'\cup\{n+1\},P_2')-(P_1',P_2'\cup\{n+1\}).
  \end{align*}
  The diagram below has exact rows; we claim it commutes.
  \begin{center}
    \begin{tikzcd}
      &\Q\bF_{n}\arrow[r,"\gamma"]\arrow[d,"\even_{n}"]&\Q\bF_{n+1}\arrow[r,"\tilde{\pi'}_*"]\arrow[d,"\odd_{n+1}"]&\Q\bF_{n}\arrow[r,""]\arrow[d,"\odd_{n}"]&0\\
0\arrow[r]&\Q\bE_{n}\arrow[r,"\alpha"]&\Q\bO_{n+1}\arrow[r,"\beta"]&\Q\bO_{n}\arrow[r]&0
    \end{tikzcd}
  \end{center}

 \noindent\textbf{Commutativity of the left square:} For
    $(P_1',P_2')\in \Q\bF_{n},$
    \begin{align*}
      \odd_{n+1}(\gamma(P_1',P_2'))&=\odd_{n+1}(P_1'\cup\{n+1\},P_2')-\odd_{n+1}(P_1',P_2'\cup\{n+1\})\\
      &=\sum_{\substack{T\subseteq P_1'\cup\{n+1\}\\\text{$\abs{T}$      odd}}}(-T)+\sum_{\substack{
        T\subseteq P_2'\\\text{$\abs{T}$
          odd}}}(T)-\left(\sum_{\substack{
        T\subseteq P_1'\\\text{$\abs{T}$
          odd}}}(-T)+\sum_{\substack{
        T\subseteq P_2'\cup\{n+1\}\\\text{$\abs{T}$
          odd}}}(T)\right)\\
      &=\sum_{\substack{T\subseteq P_1'\cup\{n+1\}\\\text{$\abs{T}$
            odd}\\
          n+1\in T}}(-T)+\sum_{\substack{T\subseteq P_2'\cup\{n+1\}\\\text{$\abs{T}$
            odd}\\
          n+1\in T}}(-T)\\
      &=\sum_{\substack{T'\subseteq P_1'\\\text{$\abs{T'}$
            even}}}(-(T'\cup\{n+1\}))+\sum_{\substack{T'\subseteq P_2'\\\text{$\abs{T'}$
            even}}}(-(T'\cup\{n+1\}))\\
      &=\alpha(\even_n(P_1',P_2')).
    \end{align*}
  \noindent\textbf{Commutativity of the right square:} For
    $(P_1,P_2)\in \Q\bF_{n+1},$
    \begin{align*}
      \odd_{n}(\tilde{\pi'}_*(P_1,P_2))&=\odd_{n}(P_1\setminus\{n+1\},P_2\setminus\{n+1\})\\
                                &=\sum_{\substack{T\subseteq P_1\setminus\{n+1\}\\\text{$\abs{T}$
      odd}}}(-T)+\sum_{\substack{T\subseteq P_2\setminus\{n+1\}\\\text{$\abs{T}$
      odd}}}(T)\\
                                &=\beta(\odd_{n+1}(P_1,P_2)).
    \end{align*}

  This proves the claim. By the inductive hypothesis, $\even_{n}$
  and $\odd_{n}$ are surjective, so by the Four Lemma, $\odd_{n+1}$ is surjective.
 
 The digram below has exact rows; we claim it commutes.
  \begin{center}
    \begin{tikzcd}
      &\Q\bF_{n}\arrow[r,"\gamma"]\arrow[d,"\odd_{n}"]&\Q\bF_{n+1}\arrow[r,"\tilde{\pi'}_*"]\arrow[d,"\even_{n+1}"]&\Q\bF_{n}\arrow[r]\arrow[d,"\even_{n}"]&0\\
      0\arrow[r]&\Q\bO_{n}\arrow[r,"\alpha"]&\Q\bE_{n+1}\arrow[r,"\beta"]&\Q\bE_{n}\arrow[r]&0
    \end{tikzcd}
  \end{center}

\noindent\textbf{Commutativity of the left square:} For
    $(P_1',P_2')\in \Q\bF_{n},$
    \begin{align*}
      \even_{n+1}(\gamma(P_1',P_2'))&=\even_{n+1}(P_1'\cup\{n+1\},P_2')-\even_{n+1}(P_1',P_2'\cup\{n+1\})\\
                               &=\sum_{\substack{T\subseteq P_1'\cup\{n+1\}\\\text{$\abs{T}$
      even}}}(-T)+\sum_{\substack{T\subseteq P_2'\\\text{$\abs{T}$
      even}}}(-T)-\left(\sum_{\substack{T\subseteq P_1'\\\text{$\abs{T}$
      even}}}(-T)+\sum_{\substack{T\subseteq P_2'\cup\{n+1\}\\\text{$\abs{T}$
      even}}}(-T)\right)\\
                               &=\sum_{\substack{T\subseteq P_1'\cup\{n+1\}\\\text{$\abs{T}$
      even}\\n+1\in T}}(-T)+\sum_{\substack{T\subseteq P_2'\cup\{n+1\}\\\text{$\abs{T}$
      even}\\n+1\in T}}(T)\\
                               &=\sum_{\substack{T'\subseteq P_1'\\\text{$\abs{T'}$
      odd}}}-(T'\cup\{n+1\})+\sum_{\substack{T'\subseteq P_2'\\\text{$\abs{T'}$
      odd}}}(T'\cup\{n+1\})\\
      &=\alpha(\odd_{n}(P_1',P_2')).
    \end{align*}
 \noindent\textbf{Commutativity of the right square:} For
    $(P_1,P_2)\in \Q\bF_{n+1},$
    \begin{align*}
      \even_{n}(\tilde{\pi'}_*(P_1,P_2))&=\even_{n}(P_1\setminus\{n+1\},P_2\setminus\{n+1\})\\
                                 &=\sum_{\substack{T\subseteq P_1\setminus\{n+1\}\\\text{$\abs{T}$
      even}}}(-T)+\sum_{\substack{T\subseteq P_2\setminus\{n+1\}\\\text{$\abs{T}$
      even}}}(-T)\\
                                 &=\beta(\even_{n+1}(P_1,P_2)).
    \end{align*}
  This proves the claim. Again, by the inductive hypothesis, $\odd_{n}$ and $\even_{n}$ are surjective, so by Lemma \ref{lem:fourvariant}, $\even_{n+1}$ is surjective.
\end{proof}

We only use the fact that $\odd_n$ is surjective to proceed; we use it to prove Lemma \ref{lem:n2surjective}

\begin{lem}\label{lem:n2surjective}
  For all $n\ge 2$, the map $\tilde{\phi}_{-1,n}:\Q\bSP_{-1,n}\to \Q\bK^{-1}_n$ is surjective.
\end{lem}
\begin{proof}
  We induct on $n$. 
  
\noindent \textbf{Base case:} $n=2.$ We have $\bSP_{-1,2}=\{\{\{1\},\{2\}\}\}$ and $\bK^{-1}_2=\{\{1,2\}\}$. We have $\tilde{\phi}_{2,2}(\{\{1\},\{2\}\})=\langle\{\{1\},\{2\}\},\{1,2\}\rangle\cdot\{1,2\}=1\cdot\{1,2\}$, which shows that $\tilde{\phi}_{2,2}$ is surjective. 
  
\noindent  \textbf{Inductive hypothesis:} The lemma holds up to some $n\ge 2$. 
  
\noindent  \textbf{Inductive step:} Define a map $\gamma:\Q\bF_{n}\to \Q\bSP_{-1,n+1}$, where $\gamma((P_1', P_2'))=\{P_1'\cup\{n+1\}, P_2'\}-\{P_1', P_2'\cup\{n+1\}\}$. Consider the digram 
  \begin{center}
    \begin{tikzcd}
      &\Q\bF_{n}\arrow[r,"\gamma"]\arrow[d,"\odd_{n}"]&\Q\bSP_{-1,n+1}\arrow[r,"\tilde{\pi}_*"]\arrow[d,"\tilde{\phi}_{-1,n+1}"]&\Q\bSP_{-1,n}\arrow[r]\arrow[d,"\tilde{\phi}_{-1,n}"]&0\\
      0\arrow[r]&\Q\bO_{n}\arrow[r,"\alpha"]&\Q\bK^{-1}_{n+1}\arrow[r,"\beta"]&\Q\bK^{-1}_{n}\arrow[r]&0
    \end{tikzcd}
  \end{center}
 Here, $\tilde{\pi}_*\circ\gamma=0$, the bottom row is exact, and $\tilde{\pi}_*$ is surjective. Note that the right square commutes by Lemma \ref{lem:alphabetacommutewithpushpull}; we claim the left square commutes as well.

 \noindent\textbf{Commutativity of the left square:} For $(P_1',P_2')\in \Q\bF_{n},$
  \begin{align*}
    \tilde{\phi}_{-1,n+1}(\gamma(P_1',P_2'))&=\tilde{\phi}_{-1,n+1}(\{P_1'\cup\{n+1\},P_2'\})-\tilde{\phi}_{-1,n+1}(\{P_1'\},P_2'\cup\{n+1\}\})\\
                                 &=\sum_{\substack{T\subseteq[n+1]\\\abs{T}\ge2\\\text{$\abs{T}$
    even}\\T\cap(P_1'\cup\{n+1\})\ne\emptyset\\T\cap P_2'\ne\emptyset}}(T)-\sum_{\substack{T\subseteq[n+1]\\\abs{T}\ge2\\\text{$\abs{T}$
    even}\\T\cap(P_1')\ne\emptyset\\T\cap
    P_2'\cup\{n+1\}\ne\emptyset}}(T)\\
                                 &=\sum_{\substack{T\subseteq[n+1]\\\abs{T}\ge2\\\text{$\abs{T}$
    even}\\n+1\in T\\T\setminus\{n+1\}\subseteq P_2'}}(T)-\sum_{\substack{T\subseteq[n+1]\\\abs{T}\ge2\\\text{$\abs{T}$
    even}\\n+1\in T\\T\setminus\{n+1\}\subseteq P_1'}}(T)\\
                                 &=\sum_{\substack{T'\subseteq[n]\\\text{$\abs{T}$
    odd}\\T'\subseteq P_2'}}(T'\cup\{n+1\})-\sum_{\substack{T'\subseteq[n]\\\text{$\abs{T}$
    odd}\\T'\subseteq P_1'}}(T'\cup\{n+1\})\\
                                 &=\alpha(\odd_{n}(P_1',P_2')).
  \end{align*}
This proves the claim. Since $\odd_{n}$ is surjective, and by the
inductive hypothesis so is $\tilde{\phi}_{-1,n}$, By Lemma \ref{lem:fourvariant}, $\tilde{\phi}_{-1,n+1}$ is surjective.
\end{proof}

\begin{lem}\label{lem:n3surjective}
  For all $n\ge 3$, the map $\tilde{\phi}_{0,n}:\Q\bSP_{0,n}\to \Q\bK^{0}_{n}$ is surjective.
\end{lem}
\begin{proof}
We induct on $n$. 

\noindent\textbf{Base case:} We have $\bSP_{0,3}=\{\{\{1\},\{2\},\{3\}\}\}$, $\bK^{0}_{3}=\{\{1,2,3\}\}$, and $\tilde{\phi}_{0,3}(\{\{1\},\{2\},\{3\}\})=\{1,2,3\}$, so $\tilde{\phi}_{0,3}$ is surjective. 

\noindent\textbf{Inductive hypothesis:} The proposition holds up to some $n\ge4$. 


\noindent\textbf{Inductive step:} Consider the following diagram, which commutes by Lemma \ref{lem:alphabetacommutewithpushpull}
\begin{center}
    \begin{tikzcd}
      &\Q\bSP_{-1,n}\arrow[r,"\tilde{\pi}^*"]\arrow[d,"\tilde{\phi}_{-1,n}"]&\Q\bSP_{0,n+1}\arrow[r,"\tilde{\pi}_*"]\arrow[d,"\tilde{\phi}_{0,n+1}"]&\Q\bSP_{0,n}\arrow[r]\arrow[d,"\tilde{\phi}_{0,n}"]&0\\
      0\arrow[r]&\Q\bK^{-1}_{n}\arrow[r,"\alpha"]&\Q\bK^{0}_{n+1}\arrow[r,"\beta"]&\Q\bK^{0}_{n}\arrow[r]&0
    \end{tikzcd}
  \end{center}
  By the inductive hypothesis, $\tilde{\phi}_{0,n}$ is
  surjective. By Lemma \ref{lem:n2surjective}, $\tilde{\phi}_{-1,n}$ is
  surjective, so by the Four Lemma, $\tilde{\phi}_{0,n+1}$ is surjective, as desired.
\end{proof}

\subsection{An inductive proof of Theorem \ref{thm:root}}\label{sec:Main}

\begin{lem}\label{lem:dimensionsequal}
  $\dim\cQ_{1,n}=\dim \Q\bK^1_{n}$ for all $n\ge 4$.
\end{lem}
\begin{proof}
  There are no Type II 1-dimensional boundary strata, so $\forall n\ge 4$, $\cV_{1,n}=\{0\}$ and $\cQ_{1,n}\cong\cA_{1,n}$. By \cite{FarkasGibney2003}, $\dim\cA_{1,n}=2^{n-1}-\binom{n}{2}-1$. On the other hand,
\begin{align*}
  \dim \Q\bK^1_{n}=\#\{T\subseteq[n]|\text{$\abs{T}$ even},\abs{T}\ge4\}=2^{n-1}-\binom{n}{2}-1.
\end{align*}
\end{proof}

\begin{lem}\label{lem:nminus3}
For all $n\ge 4$, $\phi_{n-3,n}$ is an isomorphism.
\end{lem}
\begin{proof}
For all $n\ge 4$, we have that $\cV_{n-3,n}=\{0\}$, $\cA_{n-3,n}=\cQ_{n-3,n}=\Q\{\{\{1\},\ldots,\{n\}\}\}$, $\bK^{n-3}_n=\{[n]\}$, and  $\phi_{n-3,n}:\cQ_{n-3,n}\to\Q\bK^{n-3}_{n}$ sends $\{\{1\},\ldots,\{n\}\}$ to $[n]$. Thus $\phi_{n-3,n}$ is an
  isomorphism.
\end{proof}

\begin{thm}\label{thm:root}
For $n\ge4$ and $d$ such that $1\le d\le n-3$, $\phi_{d,n}:\cQ_{d,n}\to \Q\bK^{d}_{n}$ is an isomorphism. 
\end{thm}

\begin{proof}[Inductive proof of Theorem \ref{thm:root}]

 \noindent\textbf{Base case:} $n=4$; then $1\le d\le n-3$ implies that $d=1=n-3$. By Lemma \ref{lem:nminus3}, $\phi_{1,4}$ is an isomorphism.

 \noindent\textbf{Inductive hypothesis:} For some $n\ge4,$ and for all $d$ with $1\le d\le n-3,$ we have that $\phi_{d,n}:\cQ_{d,n}\to \Q\bK^{d}_{n}$ is an isomorphism. 
 
 \noindent\textbf{Inductive step:} For $1\le d\le (n-4)$, we have the following diagram, which commutes by Lemma \ref{lem:alphabetacommutewithpushpull}:
\begin{center}
    \begin{tikzcd}
      \Ker(\pi^*)\arrow[r]\arrow[d]&\cQ_{d,n}\arrow[r,"\pi^*"]\arrow[d,"\phi_{d,n}","\cong"
      swap]&\cQ_{d+1,n+1}\arrow[r,"\pi_*"]\arrow[d,"\phi_{d+1,n+1}"]&\cQ_{d+1,n}\arrow[r]\arrow[d,"\phi_{d+1,n}","\cong" swap]&0\\
      0\arrow[r]&\Q\bK^{d}_{n}\arrow[r,"\alpha"]&\Q\bK^{d+1}_{n+1}\arrow[r,"\beta"]&\Q\bK^{d+1}_{n}\arrow[r]&0
    \end{tikzcd}
  \end{center}
  The top row is exact by Lemma \ref{lem:exactinmiddle}, and the bottom row is exact by Lemma \ref{lem:exactsequence}. By the inductive hypothesis, $\phi_{d,n}$ and $\phi_{d+1,n}$ are isomorphisms. So, by the Five Lemma, $\phi_{d+1,n+1}$ is an isomorphism. Combining the above argument with Lemma \ref{lem:nminus3}, we conclude that for $2\le d\le (n+1)-3$, $\phi_{d,n+1}$ is an isomorphism. We also have the following diagram, which commutes by Lemma \ref{lem:alphabetacommutewithpushpull}:
\begin{center}
    \begin{tikzcd}
      &\Q\bSP_{0,n}\arrow[r,"\tilde{\pi}^*"]\arrow[d,"\tilde{\phi}_{0,n}"]&\cQ_{1,n+1}\arrow[r,"\pi_*"]\arrow[d,"\phi_{1,n+1}"]&\cQ_{1,n}\arrow[r]\arrow[d,"\phi_{1,n}",
      "\cong" swap]&0\\
      0\arrow[r]&\Q\bK^{0}_{n}\arrow[r,"\alpha"]&\Q\bK^{1}_{n+1}\arrow[r,"\beta"]&\Q\bK^{1}_{n}\arrow[r]&0
    \end{tikzcd}
  \end{center}
where the bottom row is exact and the top row is a complex, exact at $\cQ_{1,n}$. By Lemma \ref{lem:n3surjective}, $\tilde{\phi}_{0,n}$ is surjective, and by the inductive hypothesis, $\phi_{1,n}$ is an isomorphism. By Lemma \ref{lem:fourvariant}), $\phi_{1,n+1}$ is surjective. But by Lemma \ref{lem:dimensionsequal}, $\dim\cQ_{1,n+1}=\dim
\Q\bK^1_{n+1}$, so $\phi_{1,n+1}$ is an isomorphism.

\end{proof}

\subsection{Theorem \ref{thm:Main} and its proof}

\begin{thm}\label{thm:Main} For $n\ge4$ and $d$ such that $1\le d\le n-3$:  
\begin{enumerate}[label=(\roman*)]
\item We have $\cK_n^d= \cV_{d,n}^{\perp}$.\label{it:KappaPerpExactly}
\item The pairing $\cQ_{d,n}\times\cK_{n}^d\to\Q$ is perfect. \label{it:KappaTypeIPairingPerfect}
\item The set $\{\kappa_{d}^T\thickspace|\thickspace\abs{T}\ge(d+3), \abs{T}\equiv (d+3)\mod 2\}$ is an $S_n$-equivariant basis for $\cK^{d}_{n}$. \label{it:KappaBasis}
\item The $S_n$ actions on $\cQ_{d,n}$ and $\cK^{d}_{n}$ are isomorphic to the permutation representation induced by the natural action of $S_n$ on the set $\{T\subseteq[n]\thickspace|\thickspace\abs{T}\ge(d+3), \abs{T}\equiv (d+3)\mod 2\}$.  \label{it:Representations}
\item The following (dual) sequences are exact:\label{it:Sequences}
    \begin{align}
      &0\to\cQ_{d,n}\xrightarrow{\pi^*}{}\cQ_{d+1,n+1}\xrightarrow{\pi_*}{}\cQ_{d+1,n}\to0\label{eq:seqQ}\\
       &0\to\cK^{d+1}_{n}\xrightarrow{\pi^*}{}\cK^{d+1}_{n+1}\xrightarrow{\pi_*}{}\cK^d_{n}\to0 \label{eq:seqK}
    \end{align}
\end{enumerate}
\end{thm}

\begin{proof}
Recall the map $\psi_{d,n}:\Q\bK^d_n\to\cK^d_n$ given in Definition \ref{Def:bKintocK}. We have compatible pairings $\cQ_{d,n}\times\cK_{n}^d\to\Q$ and $\cQ_{d,n}\times\Q\bK_{n}^d\to\Q$, inducing maps $\eta_{d,n}:\cQ_{d,n}\to(\cK_{n}^d)^{\vee}$ and $\phi_{d,n}:\cQ_{d,n}\to(\Q\bK_{n}^d)^{\vee}=\Q\bK_{n}^d$, where, $\phi_{d,n}$ is as in Definition \ref{def:phi}. These maps satisfy: $\phi_{d,n}=(\psi_{d,n})^{\vee}\circ\eta_{d,n}$. By Theorem \ref{thm:root}, $\phi_{d,n}$ is an isomorphism, which implies that $\eta_{d,n}$ is injective. On the other hand, we have by Corollary \ref{cor:kappaandQduals} that $\cK^d_n\subset\cV_{d,n}^{\perp}=(\cQ_{d,n})^{\vee}$, so $(\eta_{d,n})^{\vee}$ is injective as well. This implies that $\eta_{d,n}$ is an isomorphism, proving items \ref{it:KappaTypeIPairingPerfect} and \ref{it:KappaPerpExactly}. Since $\eta_{d,n}$ and $\phi_{d,n}$ are both isomorphisms, we conclude that so is $\psi_{d,n}$, proving item \ref{it:KappaBasis}, and thus also item \ref{it:Representations}. Finally, by Theorem \ref{thm:root} and Lemma \ref{lem:alphabetacommutewithpushpull}, the sequence in Equation \ref{eq:seqQ} is dual to the sequence in Equation \ref{eq:seqbK}, which is exact by Lemma \ref{lem:exactsequence}. We conclude that the sequence in Equation \ref{eq:seqQ} is exact. The sequence in Equation \ref{eq:seqQ} is dual to the sequence in Equation \ref{eq:seqK}, so the latter sequence is exact, proving item \ref{it:Sequences}
\end{proof}

\bibliographystyle{amsalpha}
\bibliography{../BigRefs}
\end{document}